\documentclass[12pt]{amsart}
\usepackage[utf8]{inputenc}
\usepackage{amsmath}
\usepackage{amssymb}
\usepackage{amsfonts}
\usepackage{amsthm}
\usepackage{mathrsfs} 
\usepackage{bm}
\usepackage{bbm}
\usepackage{tikz}
\usepackage{centernot}
\usepackage{hyperref}
\usepackage[margin=1.2in]{geometry}

\renewcommand{\P}{\Bbb{P}}

\newcommand{\ep}{\epsilon}

\hypersetup{
    colorlinks=true,
    linkcolor=blue,
    filecolor=magenta,
    urlcolor=blue,
    citecolor=blue
}

\makeatletter
\newtheorem*{rep@theorem}{\rep@title}
\newcommand{\newreptheorem}[2]{%
\newenvironment{rep#1}[1]{%
 \def\rep@title{#2 \ref{##1}}%
 \begin{rep@theorem}}%
 {\end{rep@theorem}}}
\makeatother

\newtheorem{thm}{Theorem}
\newreptheorem{thm}{Theorem}

\newtheorem{result}{Result}[section]

\newtheorem{prp}[result]{Proposition}

\theoremstyle{definition}

\newtheorem*{ack}{Acknowledgements}

\theoremstyle{remark}

\newcommand{\hide}[1]{}

\newcommand{\rough}[1]{}
\definecolor{darkgreen}{RGB}{75,150,75}
\newcommand{\review}[1]{}

\newcommand{\hides}[1]{}
\newcommand{\pub}[1]{}

\title{A note on large induced subgraphs with prescribed residues in bipartite graphs}
\author{Zach Hunter}
\email{zachary.hunter@exeter.ox.ac.uk}
\date{\today}

\begin{document}

\maketitle
\begin{abstract}
    It was proved by Scott that for every $k\ge2$, there exists a constant $c(k)>0$ such that for every bipartite $n$-vertex graph $G$ without isolated vertices, there exists an induced subgraph $H$ of order at least $c(k)n$ such that $\deg_H(v) \equiv 1\pmod{k}$ for each $v \in H$. Scott conjectured that $c(k) = \Omega(1/k)$, which would be tight up to the multiplicative constant. We confirm this conjecture.
\end{abstract}

\section{Introduction}

Given a graph $G$ and integers $q>r\ge 0$, we define $f(G,r,q)$ to be the maximum order of an induced subgraph $H$ of $G$ where $\deg_H(v) \equiv r \pmod{q}$ for all $v \in H$ (or if no such $H$ exists, we set $f(G,r,q) = 0$). 

There are many questions and conjectures concerning the behavior of $f(G,r,q)$ for various $G,r,q$. An old unpublished result of Gallai in this area is that\footnote{Actually what Gallai proved was slightly stronger. He showed that for each graph $G$, we can partition $V(G)$ into two parts $A,B$ so that $\deg_{G[A]}(v) \equiv 0\pmod{2}$ (respectively $\deg_{G[B]}(v) \equiv 0 \pmod{2}$) for each $v\in A$ (respectively $v\in B$).} $f(G,0,2)\ge n/2$ for every $n$-vertex graph (see \cite[Excercise~5.17]{lovasz} for a proof). Further questions about the behavior of $f$ received attention around 20-30 years ago (see e.g., \cite{caro,caro2,scott,scott2}). And more recently, this topic has had a minor renaissance (see e.g., \cite{balister,ferber,ferber2}).

This note will focus on an old result of Scott. For positive integer $k$, we define $c(k)$ to be $\inf_G\{ f(G,1,k)/|G|\}$ where $G$ ranges over all bipartite graphs with $\delta(G)\ge 1$. The following was proved by Scott:
\begin{thm}\cite[Lemma~8]{scott2} Let $k\ge 2$. Then
\[1/(2^k+k+1) \le c(k) \le  1/k.\]
\end{thm}\noindent Scott observed that a slightly more careful argument could further show that $c(k) = \Omega\left(\frac{1}{k^2 \log k}\right)$.

In this note we give an improved lower bound to $c(k)$ which is optimal up to the (implied) multiplicative constant.
\begin{thm}\label{main}Let $k \ge 2$. Then $c(k) = \Omega(1/k)$. 
\end{thm} \noindent This is done by taking the improved argument suggested by Scott, and then applying a dyadic pigeonhole argument which was previously overlooked.

\section{Proof of Theorem 2}

We will need the following result on the mixing time of random walks modulo $k$.
\begin{prp}\label{equi} Let $X_i$ be i.i.d. random variables that sample $\{0,1\}$ uniformly at random. If $n\ge k^3$, then $\P\left(\sum_{i=1}^n X_i \equiv 1 \pmod{k}\right)\ge (1-o_k(1))/k$.
\end{prp}\noindent \hide{When outlining how to prove $c(k)\ge \Omega(k^2\log k)$ in \cite{scott2}, Scott mentions that mixing results like the above follow from arguments in \cite{diaconis} (in particular this is essentially a slight modification of \cite[Theorem 2 of Chapter 3]{diaconis}). }Proposition~\ref{equi} is a mild variant of several known results, and $k^3$ could replaced with $k^2\log k$ (or any function which is $\omega(k^2)$). We omit its proof, to keep our paper short and our methods elementary. 

In \cite{scott2}, when Scott outlined how to prove $c(k) \ge \Omega\left(\frac{1}{k^2\log k}\right)$, he noted that Proposition~\ref{equi} (the key to the improvement) can be derived by slightly modifying the argument in \cite[Theorem~2 of Chapter~3]{diaconis}. These appropriate modifications now appear in \cite{ferber}. Namely, the interested reader can confirm that Proposition~\ref{equi} follows from the proof\footnote{In \cite{ferber}, the statement of their lemma hides some constants which are necessary to verify Proposition~\ref{equi}.} of \cite[Lemma~2.3]{ferber}. Both of these proofs rely on discrete Fourier Analysis.

We now proceed to the main proof.
\begin{proof}[Proof of Theorem~\ref{main}]\vspace{-\topsep} Let $G$ be an $n$-vertex bipartite graph with $\delta(G) \ge 1$, and let $V_1,V_2$ bipartition $G$ with $|V_1| \ge |V_2|$. We shall write $c_1,c_2$ to denote small positive quantities which will be determined later (it would suffice to take $c_1=1/4,c_2=1/2$, but for clarity and a slightly better constant we will only consider their values at the end of the proof and shall have them depend slightly on $k$). Our proof splits into three cases.

We take $W_1 \subset V_2$ to be a minimal set satisfying $|N(v)\cap W_1| >0$ for all $v \in V_1$ (i.e., $W_1$ is a minimal dominating set of $V_1$). By minimality of $W_1$, for each $w \in W_1$ there must exist $v_w \in V_1$ where $N(v_w) \cap W_1 = \{w\}$. Let $S_1 = \{v_w:w \in W_1\}$. We conclude that $W_1 \cup S_1$ induces a matching in $G$, proving that $f(G,1,k) \ge 2|W_1|$.

Hence, we will be done if $|W_1| \ge c_1 |V_1|/k$ (this is ``Case 1''). So we continue assuming $|W_1| < c_1 |V_1|/k$. 

For $2\le i \le k-1$, we inductively create sets $W_i,S_i$. We take $W_i\subset W_{i-1}$ to be a minimal dominating set of $V_1\setminus \left(\bigcup_{j=1}^{i-1} S_j\right)$. And like in the above, we take $S_i \subset V_1\setminus \left(\bigcup_{j=1}^{i-1} S_j\right)$ so that $W_i \cup S_i$ induces a matching in $G$.

Let $T = V_1 \setminus \left(\bigcup_{i=1}^{k-1} S_i \right)$. We have
\begin{align*}
    |T| &= |V_1|- \sum_{i=1}^{k-1} |S_i|\\
    &= |V_1|- \sum_{i=1}^{k-1} |W_i|\\
    &\ge  |V_1|- (k-1)|W_1|\\
    &\ge (1-c_1)|V_1|.\\
\end{align*}

Next, let $T^* = \{v \in T: |N(v)\cap W_{k-1}| \ge k^3\}$. Supposing that $|T^*| \ge c_2 |V_1|$ (this is ``Case 2''), we will deduce that $f(G,1,k)\ge (c_2-o_k(1))|V_1|/k$. 

Indeed, let $U\subset W_{k-1}$ be a random subset where each element is included (independently) with probability $1/2$. We set $T_U = \{v \in T: |N(v)\cap U| \equiv 1 \pmod{k}\}$. By Proposition~\ref{equi}, we have that $\P(v \in T_U) \ge (1-o_k(1))/k$ for each $v \in T^*$. Thus by linearity of expectation we may fix some $U\subset W_{k-1}$ where $|T_U| \ge |T^*|(1-o_k(1))/k \ge (c_2-o_k(1))|V_1|/k$. Next choosing $S \subset \bigcup_{i=1}^{k-1}S_i$ so that $|N(u)\cap (T_U\cup S)| \equiv 1 \pmod{k}$ for each $u \in U$, we have that $S\cup U \cup T_U$ induces a subgraph in $G$ demonstrating that $f(G,1,k) \ge |S\cup U \cup T_U| \ge |T_U| \ge (c_2-o_k(1))|V_1|/k$.

Otherwise, we must have that $T\setminus T^*$, the set of $v \in T$ where $|N(v)\cap W_{k-1}| <k^3$, has $>(1-c_1-c_2)|V_1|$ elements (this is ``Case 3''). By dyadic pigeonhole, there exists some $0 \le p\le \log(k^3 ) = O(\log k)$ so that \begin{align*}
    |\{v \in T : 2^p \le  |N(v)\cap W_{k-1}| <2^{p+1}\}| &\ge |T\setminus T^*|/O(\log k) \\
    &\ge (1-c_1-c_2)|V_1|/O(\log k).\\
\end{align*}Take $T' = \{v \in T : 2^p \le  |N(v)\cap W_{k-1}| <2^{p+1}\}$ to be this large set.

We let $U \subset W_{k-1}$ be a random subset so that each element is included (independently) with probability $1/2^p$. Defining $T_U$ as before, some casework\footnote{If $p =0$, then $U = W_{k-1}$ and this probability is one. Otherwise this probability is $\binom{|N(v)\cap W_{k-1}|}{1} (1-2^{-p})^{|N(v)\cap W_{k-1}|} 2^{-p} \ge (1-2^{-p})^{2^{p+1}-1} \ge e^{-2}$.} shows $\P(v \in T_U) \ge e^{-2}$ for each $v \in T'$. Hence, by linearity of expectation, we may fix $U$ so that $|T_U|\ge  e^{-2}|T'|$. As above we may find $S \subset \bigcup_{i=1}^{k-1} S_i$ so that $S\cup U\cup T_U$ demonstrates that $f(G,1,k) \ge |S\cup U \cup T_U| \ge e^{-2}(1-c_1-c_2) |V_1| /O(\log k)$.

Now fix any sufficiently small $\ep > 0$. Letting $c_1 = 1/3-\ep/2, c_2 = 2/3-\ep$, we get that each of the first two cases imply that $f(G,1,k)\ge (2/3-\ep-o_k(1))|V_1|/k \ge (1/3-\ep-o_k(1))n/k$ (since $|V_1| \ge |V_2|$). Meanwhile with $\ep$ fixed, the third case implies $f(G,1,k) = \Omega_{\ep}(n/\log k)$. Taking $\ep \downarrow 0$ as $k \to \infty$ we have that $f(G,1,k) \ge (1/3-o_k(1))n/k$.
\end{proof}

As a closing remark, we note it is still open whether $c(k) =1/k$ for all $k$ (as noted in \cite{scott2}, considering $K_{k,k}$ demonstrates that $c(k) \le 1/k$). Even for $k=2$, the best known bounds are $1/4 \le c(2)\le 1/2$, with the lower bound coming from \cite[Theorem~2]{scott}.

\begin{ack}The author thanks Zachary Chase for spotting some typographical errors in a previous draft of this paper.
\end{ack}

\end{document}